\documentclass[11pt,reqno]{amsart}

\usepackage[margin=25mm,a4paper]{geometry}
\usepackage{mathtools}
\usepackage{amsaddr}
\usepackage{bbm}
\usepackage[shortlabels]{enumitem}
\usepackage{cite}
\usepackage{graphicx}      
\usepackage[labelfont=bf,labelformat=simple]{subcaption}
	
\usepackage{hyperref}

\theoremstyle{definition}
\newtheorem{definition}{Definition}[section]
\theoremstyle{plain}
\newtheorem{theorem}[definition]{Theorem}

\newtheorem{lemma}[definition]{Lemma}
\newtheorem{corollary}[definition]{Corollary}
\theoremstyle{remark}
\newtheorem{remark}[definition]{Remark}

\begin{document}

\title{Entropy of Wiener integrals with~respect~to~fractional~Brownian~motion}

\author{Iryna Bodnarchuk$^1$, Yuliya Mishura$^1$, Kostiantyn Ralchenko$^{1,2}$}

\address{$^1$Taras Shevchenko National University of Kyiv\\
$^2$University of Vaasa}

\thanks{YM is supported by The Swedish Foundation for Strategic Research, grant UKR24-0004, and by the Japan Science and
Technology Agency CREST, project reference number JPMJCR2115.
KR is supported by the Research Council of Finland, decision number 359815.
YM and KR acknowledge that the present research is carried out within the frame and support of the ToppForsk project no.~274410 of the Research Council of Norway with the title STORM: Stochastics for Time-Space Risk Models.}

\begin{abstract}
The paper is devoted to the properties of the entropy of the exponent-Wiener-integral fractional Gaussian process (EWIFG-process), that is a Wiener integral of the exponent with respect to fractional Brownian motion. Unlike fractional Brownian motion, whose entropy has very simple monotonicity properties in Hurst index, the behavior of the entropy of EWIFG-process is much more involved and depends on the moment of time. We consider these properties of monotonicity in great detail.
\end{abstract}

\maketitle

\section{Introduction}\label{sec:Intr}
Thousands of papers have been devoted to the concept of thermodynamic entropy, as well as to the notions of Shannon, Rényi, and various other types of entropy for discrete probability distributions and distributions with densities. Being mainly involved in mathematics, we only mention papers and books \cite{campbell1972characterization,feder1994relations,martin2011mathematical,papalexiou2012entropy,renyi1961measures,ribeiro2021entropy,vancam} without in the least claiming that the given list of citations is complete.

The reason for this  interest is that  the concept of entropy is interesting both in itself and from the point of view of numerous applications. And if we study the entropies of certain probability distributions, we need to know the properties of the entropies as functions of the distribution parameters. As one example, the paper \cite{MMRR2023} took a step towards systematizing the properties of different entropies calculated for the normal distribution. It was established  that Shannon, R\'{e}nyi, one- and two-parameter generalized R\'{e}nyi, Tsallis and Sharma-Mittal entropies of the centered normal distribution  are proportional (which is quite natural) to the logarithm of the  variance of the distribution. As one example, Shannon entropy of the centered normal distribution with variance $\sigma^2$ equals $$H_S=\frac12(1+\log(2\pi))+\log \sigma.$$ The next, also very natural question is to calculate entropies of   one-dimensional or multi-dimensional distributions  of various Gaussian processes and to study their properties. These problems were the subject of articles \cite{MMRR2023} and \cite{MishuraRalchenkoSchilling2022} In the first one, we   calculated six entropies of one-dimensional distributions of fractional, subfractional,  bifractional, multifractional Brownian motions and  tempered fractional Brownian motion, while in the second one, we followed the results of the book \cite{stratonovich} and  studied the properties of Shannon entropy of vector-valued fractional Gaussian noise, taken in the integer moments of time. The main hypothesis that   entropy increases when the Hurst index increases from zero to $\frac12$, and decreases when the Hurst index increases from $\frac12$  to 1, in the general case has been proven numerically  and still awaits an analytical proof.  In the present paper, we decided to study in its entirety the properties of the entropies of the fractional Gaussian process that is a Wiener integral of exponent w.r.t. the fractional Brownian motion, $X_t^H=\int_0^t e^sdB_s^H$, as functions of two parameters, $t\ge 0$ and $H\in[0, 1]$. It is interesting that unlike the fractional Brownian motion $B^H$ with Hurst index $H\in(0,1)$  whose variance at moment $t$ equals simply $t^{2H}$ and, therefore, its behavior in $t$ and $H$ is evident, the presence of exponential integrand leads to much more exotic behavior. We chose exponential integrand by three reasons: it is more interesting than power integrand for which the behavior of the variance of the respective Wiener integral is more or less trivial; it  allows us to draw some non-obvious conclusions;   it participates in the representation of fractional Ornstein--Uhlenbeck process. 

The paper is organized as follows. Section \ref{sec:Prelim} introduces our Wiener integral, that we call  ``exponent-Wiener-integral fractional Gaussian process'' (EWIFG-process), together with its square characteristics. We extend it to the cases $H=0$ and $H=1$ and also give square characteristics of these ``boundary'' processes.  Section \ref{IBYMsecfunc} contains the main results describing the monotonic behavior of the entropy of EWIFG-process as the function of $H\in[0,1]$ and $t>0.$ If in time it strictly increasing, the behavior in $H$ is unexpectedly much more involved and depends on the moment of time. We calculated two  moments of time that change the monotonicity, numerically as the solutions of transcendent equations. In Section \ref{IBYMsecsec4} we generalize some of the results to more general exponent $e^{kt}$, emphasizing on $k=-1$. Appendix provides an auxiliary result.

\section{Preliminaries}\label{sec:Prelim} 
\subsection{Fractional Brownian motion with Hurst index  \texorpdfstring{$H\in[0,1]$}{H from [0,1]}; some  properties}  Recall that the fractional Brownian motion (fBm) with Hurst index $H\in(0,1)$ is a zero mean Gaussian process $B^H=\{B_t^H, t\ge 0\}$ with covariance function $$\mathsf{R}_{B^H}(t, s)= \mathsf{E} B_t^HB_s^H=\frac12(t^{2H}+s^{2H}-|t-s|^{2H}),$$ and, consequently, with variance $\mathsf{E}(B_t^H)^2=t^{2H}.$ It is well known (see e.g., \cite{Mish-book})
that the trajectories of fBm with Hurst index $H$ are a.s.\ H\"{o}lder continuous up to order $H$.
\begin{remark}\label{rem-1}
    It is possible to extend the notion of fBm to the values $H=1$ and $H=0.$ Namely, as $H\to 1$, covariance $\mathsf{R}_{B^H}(t, s)$ tends to $ts$ for any $t, s\ge 0$, which means that, by continuity, it is natural to put $B_t^1=t\xi$, where $\xi$ is a standard normal variable. Also, in the paper \cite{novikov} fBm was extended to the value  $H=0$, by continuity of covariance function, and, as the result, fBm with $H=0$ has a form  $B_t^0=\frac{\xi_t-\xi_0}{\sqrt{2}}$, where $\xi_t$ is a white noise, i.e., the set of independent standard normal variables. Note, however, that in this case the variance of $B_t^0$  has discontinuity at zero, because $\mathsf{E}(B_t^0)^2=1$ for $t>0$ and is zero for $t=0$.
\end{remark}
\subsection{Wiener integrals w.r.t.\ fBm and their square characteristics}
According to the paper \cite{zahle}, if to consider two measurable nonrandom functions $f, g\colon[0,T]\to \mathbb{R}$ such that $f$ is H\"{o}lder up to order $\alpha>0$,  $g$ is H\"{o}lder up to order $\beta>0$, and $\alpha+\beta>1$, then the integral $\int_0^Tf(s)dg(s)$ exists as the generalized Lebesgue--Stieltjes integral and, moreover,   equals  the limit of the Riemann integral sums. 

Therefore, for any $H\in(0,1)$   the following Wiener integral with respect to (w.r.t.) fractional Brownian motion is well defined:
\begin{gather}\label{eq:Xt}
X_t^H=\int_0^t e^sdB_s^H. 
\end{gather}
However, if for $H>1/2$ representation \eqref{eq:Xt} is convenient for calculations of  its characteristics, 
for $0<H<1/2$ it is better to integrate by parts and consider the representation 
\begin{gather}\label{repres}
 X_t^H=e^{t}B_t^H-\int_0^t e^{s}B_s^Hds.
\end{gather}
\begin{remark}\label{remrem1} Of course, this representation is valid for any $H\in(0,1)$ as well, despite traditionally it is used for $0<H<1/2$.
    \end{remark} For the properties of the respective representations see also \cite{maje}. In order to distinguish $X^H$ from other Gaussian processes, we shall call it ``exponent-Wiener-integral fractional Gaussian process (EWIFG-process)''.
Now, let us calculate and simplify covariance and variance of $X^H$ for $H\in(0,1)$. In this connection, recall that for $H\in(1/2, 1)$, any $T_1, T_2>0$ and bounded, sufficiently smooth, for example,  differentiable function $f\colon [0,T_1\vee T_2]\to\mathbb{R}$ we have the equality, see \cite{valk},
\begin{equation}\label{covarvar}
\mathsf{R}_{X^H}(T_1,T_2) = \mathsf{E} X^H_{T_1}X^H_{T_2}
= H(2H-1)\int_0^{T_1}\!\!\int_0^{T_2}f(u)f(v)|u-v|^{2H-2}dudv.
\end{equation}
 Taking this into account, we can formulate and prove the next lemma (in principle, this result is well-known, however, we add it for the reader's convenience).
 \begin{lemma}\label{lem:covar}
1) For $0<H<1 $ and $s, t\ge 0$ the covariance function $\mathsf{R}_{X^H}(t, s)$ of the process $X^H$ can be presented as 
\begin{equation}\label{eq:RXH<1/2base}
\begin{split}
 \mathsf{R}_{X^H}(t, s) = \mathsf{E} X_t^HX_s^H &= \frac{1}{2}
 e^{t+s}\left(t^{2H}+s^{2H} - |t-s|^{2H}\right)
 \\
&\quad -\frac{1}{2}e^{s}\int_0^{t }e^{u}(u^{2H} +s^{2H}-|s-u|^{2H})du
 \\
&\quad -\frac{1}{2}e^{t}\int_{0}^{s}e^{v}(t^{2H}+v^{2H}-|t-v|^{2H})dv
 \\
&\quad  +\frac{1}{2}\int_{0}^{t}\!\!\int_{0}^{s}e^{u+v}(u^{2H} +v^{2H}-|v-u|^{2H})dudv.
\end{split}\end{equation}
In particular, the variance of $X_t^H,\ t\ge 0$, can be calculated as
\begin{gather}\label{varnevar1}
 \mathsf{V}(H, t):=\mathsf{E} \left(X_t^H\right)^2=e^{t}t^{2H}+\frac{1}{2}e^{2t}\int_0^te^{-z}    z^{2H}dz
  -\frac{1}{2} \int_0^t e^{z} z^{2H}dz.
\end{gather}

2) For $1/2<H<1$ and $s, t\ge 0$ the covariance function $\mathsf{R}_{X^H}(t,s)$ of the process $X^H$ can be presented as 
\begin{gather}\label{covarnevar1}
 \mathsf{R}_{X^H}(t, s)=\mathsf{E} X_t^HX_s^H=H(2H-1)\int_0^t\!\!\int_0^s e^{u+v}|u-v|^{2H-2}dudv.   
\end{gather}

In particular, the variance of $X_t,\ t\ge 0$, can be calculated as
\begin{gather}\label{eq:varbase}
 \mathsf{V}(H, t):=\mathsf{E} \left(X_t^H\right)^2
 =H(2H-1)\int_0^t\!\!\int_0^t e^{u+v}|u-v|^{2H-2}dudv.
\end{gather}
\end{lemma}
\begin{proof} Equality \eqref{eq:RXH<1/2base} follows directly from the representation \eqref{repres}, while \eqref{covarnevar1} and \eqref{eq:varbase} are direct consequences of~\eqref{covarvar}. Therefore, it is necessary to prove only \eqref{varnevar1}. However, for $0<H<1/2$
\begin{align*}
\mathsf{V}(H, t) &= e^{2t}t^{2H}-e^t\int_0^t e^v(t^{2H}+v^{2H}-(t-v)^{2H})dv\\
&\quad+\frac12 \int_0^t \!\!\int_0^t e^{v+u}(u^{2H}+v^{2H}-|v-u|^{2H})dudv\\
&=e^{2t}t^{2H}-e^t t^{2H}(e^t-1)-e^t\int_0^t e^v v^{2H}dv
+e^{2t}\int_0^t e^{-v} v^{2H} dv
\\
&\quad+(e^t -1)\int_0^t e^{u} u^{2H} du
-\frac12 \int_0^t e^v\int_0^v e^{u}(v-u)^{2H})dudv
\\
&\quad-\frac12 \int_0^t e^v\int_v^t e^{u}(u-v)^{2H}dudv\\
&=e^{t}t^{2H}+e^{2t}\int_0^t e^{-v} v^{2H} dv
-\int_0^t e^{u} u^{2H} du
\\
&\quad-\int_0^t e^v\int_0^v e^{u}(v-u)^{2H}dudv.
\end{align*}

Now,
\begin{align*}
\MoveEqLeft
\int_0^t e^v\int_0^v e^{u}(v-u)^{2H}dudv
=\int_0^t e^{2v}\int_0^v e^{-u}u^{2H}dudv\\
&=\int_0^t e^{-u}u^{2H}\int_u^t e^{2v}dvdu
=\frac12 e^{2t}\int_0^t e^{-u} u^{2H}du
-\frac12 \int_0^t e^{u} u^{2H}du,
\end{align*}
and the proof follows.
\end{proof}
\subsection{Continuity of the variance of EWIFG-process \texorpdfstring{$\int_0^te^s dB_s^H$ at $H=1/2$}{at H=1/2}}

Let $H=\frac12.$ Then $B^{1/2}=W $ is a standard Wiener process, and  $X_t^{1/2}=\int_0^te^sdW_s$ is a standard Wiener integral with variance $\mathsf{V}(1/2, t)=\int_0^te^{2s}ds=\frac{e^{2t}-1}{2}.$ This case is, in some sense, exceptional for formula \eqref{eq:varbase}, because it can not be applied directly to $H=\frac12$, but, if  we are interested in continuity of $\mathsf{V}(H, t)$ at the point $H=1/2$, we apply formula \eqref{varnevar1} and immediately get the following result.
\begin{lemma}\label{lem:covarassympt}  For any $t\ge 0$ variance $\mathsf{V}(H, t)$ is continuous at the point $H=1/2$.
    
\end{lemma}

\begin{proof}
Continuity   immediately follows from Remark \ref{remrem1}, equality \eqref {varnevar1} and the fact that for any $t\ge 0$ 
\begin{align*}
 \mathsf{V}(H, t) &= e^{t}t^{2H}+\frac{1}{2}e^{2t}\int_0^te^{-z}    z^{2H}dz
  -\frac{1}{2} \int_0^t e^{z} z^{2H}dz\\
  &\to e^{t}t+\frac{1}{2}e^{2t}\int_0^te^{-z}zdz-\frac{1}{2} \int_0^t e^{z} zdz\\
  &=e^tt+\frac12e^{2t}(1-e^{-t})-\frac12e^{t}t-\frac12(1-e^{t}+te^t)=\frac{e^{2t}-1}{2}, \quad H\to \frac12.
\end{align*}
\end{proof}

\subsection{Extension of EWIFG-process \texorpdfstring{$\int_0^te^s dB_s^H$}{} and its square characteristics to the cases \texorpdfstring{$H=1$ and $H=0$}{H=1 and H=0}}\label{sec:H=0}
Now our aim is to extend   $X^H=\{X_t^H,\ t\ge 0\}$, $H\in( 0,1)$ to the cases $H=1$ and $H=0$. 
\subsubsection{Wiener integral w.r.t. fBm with \texorpdfstring{$H=1$}{H=1}}\label{extendH=1}  This case is comparatively simple. Indeed, if $H\to 1$, then, according to equality \eqref{covarnevar1}, 
\[
\mathsf{R}_{X^H}(t, s)\to \int_0^t\!\!\int_0^se^{u+v} dudv=(e^{t}-1)(e^{s}-1),\,\, t, s\ge 0.
\]
Consequently, finite-dimensional distributions of $X^H_t$ weakly converge to finite-dimensional distributions of the Gaussian process $X^1_t=\xi\int_0^te^{u} du=\xi (e^{t}-1),\, t \ge 0,$ where $\xi$ is a standard normal variable. Obviously, the variance of $X^1_t$ equals 
\begin{gather}\label{eq:varH=1}
   \mathsf{V}(1, t):=(e^{t}-1)^2. 
\end{gather}
This way to find what is $X^1_t$ can be characterized as indirect, but going directly and substituting $B^1$ into formula \eqref{eq:Xt}, we get   $X^1_t=\int_0^t e^u dB^1_u=\xi\int_0^t e^u du=\xi (e^{t}-1)$, $t\ge 0$,   receiving the same result. 

\subsubsection{Wiener integral w.r.t. fBm with $H=0$}\label{extendH=0} This case is more involved. As it was mentioned in Remark \ref{rem-1}, for $H=0$, $B_t^0=\frac{\xi_t-\xi_0}{\sqrt{2}}$, where $\xi_t$ is a white noise. Substituting it formally into representation \eqref{repres}, we get the equality
\begin{align*}
 X_t&=e^t\frac{\xi_t-\xi_0}{\sqrt{2}}
 -\int_0^te^s\frac{\xi_s-\xi_0}{\sqrt{2}}ds
 =\frac{1}{\sqrt{2}}\left(
 e^t\xi_t-e^t\xi_0-\int_0^te^s\xi_s ds+\xi_0(e^t-1)
 \right)\\
 &=\frac{1}{\sqrt{2}}\left(
 e^t\xi_t-\int_0^te^s\xi_s ds-\xi_0 \right),
\end{align*}
which is unfortunately meaningless. Indeed, one can not even guarantee the measurability of trajectories of $\xi_s$ w.r.t. $s\ge0$, and moreover, trying to calculate  the variance of the integral $\int_0^te^s\xi_s ds$, we get an absurd result
\[
 \mathsf{E} \left(\int_0^te^s\xi_s ds\right)^2=\int_0^t\!\!\int_0^t\mathsf{E}\xi_u\xi_vdudv=0.
\]

Therefore we go by indirect way and  find the limit of covariance function $\mathsf{R}_{X^H}(t, s)$ as $H\downarrow 0$. First, we prove the following auxiliary result. 

\begin{lemma}
Let the function  $\mathsf{R}^0(t, s)$ equal 
    \begin{gather}
     \mathsf{R}^0(t, s)=\begin{cases}
         0, & t\wedge s=0,\\
         \frac{1}{2}, & t, s>0,\ t\neq s,\\
         \frac{1}{2}\left(1+e^{2t}\right), & t=s>0.
     \end{cases}   \label{eq:R0}
    \end{gather}
Then $\mathsf{R}^0(t, s), t,s\ge 0$ is  a covariance function , i.e., it is symmetric and positive definite.
\end{lemma}
\begin{proof} Symmetry is evident. 
Now, let $k>1,\ 0<t_1<\ldots<t_k,\ z_i \in\mathbb{R}, 1\le i\le k$.
Then
\begin{gather*}
    \sum_{i,j=1}^k R^0(t_i,t_j)z_iz_j
    =\frac12\sum_{i,j=1}^k z_iz_j
    +\frac12\sum_{i=1}^k e^{2t_i}z_i
    =\frac12\left(\sum_{i=1}^k z_i\right)^2
    + \frac12\sum_{i=1}^k e^{2t_i}z_i^2\ge 0.
\end{gather*}

The cases when some of $t_i$ are equal or their minimum equals  zero, are considered similarly.
\end{proof}

\begin{corollary}
    There exists a zero mean Gaussian process
    $\{X^0_t,\ t\ge 0\}$ with covariance function $\mathsf{R}^0(t, s)$, defined by formula \eqref{eq:R0}, i.e., such that
    $\mathsf{E} X_t^0 X_s^0=\mathsf{R}^0(t,s),\ t, s\ge 0$.
\end{corollary}

\begin{theorem}\label{theorem:R0}
For any $0<H<\frac12$, consider the process
$X^H=\{X_t^H,\ t\ge 0\}$ defined by \eqref{repres}.
As $H\downarrow0$, the finite-dimensional distributions of $X^H$ converge weakly to those of a Gaussian process
$X^0=\{X_t^0,\ t\ge 0\}$ with zero mean and covariance function
$\mathsf{R}^0(t, s)$, given by \eqref{eq:R0}. This means, in particular, that is natural to extend
$X^H_t,\ 0<H<1/2$ to $H=0$, by assigning $X^H\big|_{H=0}$ the value $X^0$. In other words, $X^0$ with covariance function $\mathsf{R}^0(t, s)$ and zero mean is the extension of $X^H$ to the case $H=0.$
\end{theorem}
\begin{proof}
It is sufficient to establish the point-wise convergence of covariance functions. Letting $H\downarrow 0$ in \eqref{eq:RXH<1/2base}, we get that for any $t, s>0$, $t\neq s$,
\begin{align*}
  \mathsf{R}_{X^H}(t, s)
  &\to \frac12e^{t+s}-\frac{e^{s}}{2}\int_0^t e^{u} du
  -\frac{e^{t}}{2}\int_0^s e^{v} dv
  +\frac{1}{2}\int_0^t e^{u} du \int_0^s e^{v} dv
   =\frac{1}{2}.
\end{align*}

Let $t=s>0$. Then
\begin{align*}
    \mathsf{V}(H, t)&\to e^{2t}-e^t \int_0^t e^{u} du
    + \frac12\left(\int_0^t e^{u} du\right)^2
    \\
    &=e^t+\frac12(e^{2t}-2e^t+1)=\frac{e^{2t}+1}{2}, 
    \quad H\downarrow 0.
\end{align*}

The case $t\wedge s=0$ is obvious, and  the proof follows.
\end{proof}
\begin{remark}
Covariance of $X^0$ has   discontinuities on the    axis. However, these discontinuities are a bit different: at any point $(t, s)$ with $t\wedge s=0$ and $t\vee s>0$ the limit equals $1/2$ and the value is zero, but at the origin there is no limit.  Variance has discontinuity at zero because the limit at zero equals 1 and value is zero.

\end{remark}

\section{Entropy of the  EWIFG-process  as the function of \texorpdfstring{$H\in[0,1]$ and $t>0$}{H and t}}\label{IBYMsecfunc}
 
Let us investigate the behavior of the variance of the process $X^H$ defined by \eqref{eq:Xt} as the function of $H\in[0, 1]$ for different $t\ge 0$. Entropy of its one-dimensional distribution, being proportional to the logarithm of the variance, follows the same monotonic properties. Therefore, when formulating statements about the monotonicity of variance, we simultaneously obtain the same properties of the monotonicity of entropy, for example, Shannon entropy.
\subsection{Entropy of the  EWIFG-process  as the function of \texorpdfstring{$H\in[1/2,1]$}{H from [1/2,1]}} We start with
the case $H\in[1/2,1]$.
As the first easy step, compare $\mathsf{V}(1/2, t)$ and
$\mathsf{V}(1, t)$.
 \begin{lemma}\label{lem:V1/2-V1}
For $H\in[1/2,1]$ variance $\mathsf{V}(H,t)$ has the following properties:
\begin{enumerate}[(i)]
\item $\mathsf{V}(1/2, 0)=\mathsf{V}(1, 0)=0$,
\item $\mathsf{V}(1/2, \log 3)=\mathsf{V}(1, \log 3)=4$,
\item $\mathsf{V}(1/2, t)>\mathsf{V}(1, t)$ for $0< t< \log 3$,
\item $\mathsf{V}(1/2, t)<\mathsf{V}(1, t)$ for $t> \log 3$.
\end{enumerate}
 \end{lemma}
\begin{proof}
By \eqref{eq:varH=1}, $\mathsf{V}(1, t):=(e^{t}-1)^2.$
For $H= 1/2$ we have $X_t=\int_0^t e^sdW_s$ and
$$
\mathsf{V}(1/2, t)=\int_0^t e^{2s}ds=\frac{e^{2t}-1}{2}.
$$

$(i)$--$(ii)$ The statements are easily obtained by direct substitution $t=0$ and $ t= \log 3$.

$(iii)$--$(iv)$ Inequality $\frac{e^{2t}-1}{2}> (e^{t}-1)^2$ is equivalent to the following one: 
$e^{2t}-4e^{t}+3< 0$, which is true for
$1< e^t< 3$, i.e., $0< t< \log 3$. Consequently, the opposite relationship between variances is valid for $t> \log 3$.
\end{proof}

Now we proceed with the first main result that characterizes   the behavior of the variance $\mathsf{V}(H, t)=\mathsf{E} (X_t^H)^2,\ H\in[1/2, 1]$ as the function of $H$ for different $t>0$.
\begin{theorem}\label{IBYMtheorem_2}
The behavior of $\mathsf{V}(H, t)=\mathsf{E} (X_t^H)^2,\ H\in[1/2, 1]$ is the following:
there exist two points $\tau_{1/2}>\log 3 >\tau_{1}>1$ such that
\begin{enumerate}[(a)]
    \item for any fixed $  t\in(0, \tau_1]$ it holds that  $\mathsf{V}(H, t)$ decreases in $H\in[1/2, 1]$;
    \item for any fixed $ t\in (\tau_1,  \log 3)$ there exists a value of Hurst index $1/2<H_t<1$  such that $\mathsf{V}(H, t)$ decreases in
    $H\in[1/2, H_t]$, increases in $H\in[H_t, 1]$, and
    $\mathsf{V}(1, t)<\mathsf{V}(1/2, t)$;
    \item for $t= \log 3$ there exists $1/2<H_{\log 3}<1$  such that $\mathsf{V}(\log 3, H)$ decreases in
    $H\in[1/2, H_{\log 3}]$, increases in $H\in[H_{\log 3}, 1]$, and
    $\mathsf{V}(1/2, \log 3)=\mathsf{V}(1, \log 3)$;
    \item for any fixed $ t\in(\log 3, \tau_{1/2}]$ it holds that $\mathsf{V}(H, t)$ decreases in
    $H\in[1/2, H_{t}]$, increases in $H\in[H_{t}, 1]$, and
    $\mathsf{V}(1/2, t)<\mathsf{V}(1, t)$;
    \item for $t>\tau_{1/2}$ $\mathsf{V}(H, t)$ increases in $H\in[1/2, 1]$.
\end{enumerate}
\end{theorem}
\begin{proof}
Let $H\in[1/2, 1],\,t>0$. Consider  the representation of the variance 
\begin{gather*}
    \mathsf{V}(H, t)= e^{t}t^{2H}
    +\frac12\int_0^t\left(e^{2t-z}-e^{z}\right)z^{2H}dz,
\end{gather*}
and note  that $e^{2t-z}-e^{z}\ge 0$ for $0\le z\le t$.
Consider the derivative of the variance in $H$:
\begin{gather}\label{deriv}
    \frac{\partial \mathsf{V}(H, t)}{\partial H}
    =2e^{t}t^{2H}\log t
    +\int_0^t\left(e^{2t-z}-e^{z}\right)z^{2H}\log{z}\,dz.
\end{gather}
Moreover,
\begin{gather}\label{secder}
    \frac{\partial^2 \mathsf{V}(H, t)}{\partial^2 H}
    =4e^{t}t^{2H}\log^2 t
    +2\int_0^t\left(e^{2t-z}-e^{z}\right)z^{2H}\log^2{z}\,dz>0,
    \quad t>0.
\end{gather}

It obviously means that $\frac{\partial \mathsf{V}(H, t)}{\partial H}$ strictly  increases in $H\in[1/2, 1]$. Therefore, we can consider several  cases and establish that all of them are realized.

1) For some $t>0$ $\frac{\partial \mathsf{V}(1, t)}{\partial H}< 0$. Then $\frac{\partial \mathsf{V}(H, t)}{\partial H}< 0$ for all $H\in[1/2,1]$, and $\mathsf{V}(H, t)$ decreases in $H\in[1/2,1]$.

2) For some $t>0$ $\frac{\partial \mathsf{V}(1, t)}{\partial H}\ge  0$ but $\frac{\partial \mathsf{V}(1/2, t)}{\partial H}< 0$.
It means that there exists a value of Hurst index $1/2<H_t<1$  such that $\mathsf{V}(H, t)$ decreases in $H$ for $H\in[1/2,H_t]$ and increases for $H\in[H_t,1]$.

3) For some $t>0$ $\frac{\partial \mathsf{V}(1/2, t)}{\partial H}\ge   0$, then $\frac{\partial \mathsf{V}(H, t)}{\partial H}> 0$ for all $H\in[1/2,1]$.

So, let $t>0$ be fixed. Note that
\begin{equation}\label{derV1}
   \frac{\partial \mathsf{V}(1, t)}{\partial H}
   = 2e^t t^2 \log{t}
   +\int_0^t\left(e^{2t-z}-e^{z}\right)z^2 \log{z}\, dz.
\end{equation}
Obviously, $\frac{\partial \mathsf{V}(1, t)}{\partial H}< 0$ for all $t\le 1$.
Furthermore, for $t\le 1$ consider
\begin{gather}\label{reprrepr1}
   \frac{\partial^2 \mathsf{V}(1, t)}{\partial H \partial t}
   = 2e^t t^2 \log{t} +4e^t t \log{t} + 2e^t t
   +2e^{2t}\int_0^t e^{-z} z^2 \log{z}\, dz.
\end{gather}

Let $t=1$. Then
$$
\frac{\partial^2 \mathsf{V}(1, 1)}{\partial H \partial t}
=2e+2e^{2}\int_0^1 e^{-z} z^2 \log{z}\, dz
=2e\left(1+e\int_0^1 e^{-z} z^2 \log{z}\, dz\right).
$$
Note that function $f(z)=z^2\log{z}$ equals zero at 0 and 1 and achieves its minimum at the point $z=e^{-1/2}$.
This minimal value equals $-\frac{1}{2}e^{-1}$.
Therefore
$$
1+e\int_0^1 e^{-z} z^2 \log{z}\, dz
>1+e\left(-\frac{1}{2}e^{-1}\right)\left(1-e^{-1}\right)
=\frac{1}{2}+\frac{1}{2}e^{-1}>0.
$$
Moreover, we can rewrite \eqref{reprrepr1} as 
$$
\frac{\partial^2 \mathsf{V}(1, t)}{\partial H \partial t}
=2e^{2t}\left(e^{-t} t^2\log{t}
+2e^{-t} t \log{t} + e^{-t} t
+\int_0^t e^{-z} z^2 \log{z}\, dz
\right),
$$
and this representation allows to establish that
$\frac{\partial^2 \mathsf{V}(1, t)}{\partial H \partial t}$ increases in $t$, because the value in the brackets is positive at the point $t=1$, and its derivative equals
\begin{gather*}
    -e^{-t}t^2\log{t}+2te^{-t} \log{t} + e^{-t} t
    -2e^{-t}t \log{t}+2e^{-t} \log{t}
    +3e^{-t}>0,\quad t\ge 1.
\end{gather*}

Finally, we establish that
$\frac{\partial \mathsf{V}(1, t)}{\partial H}$ increases in $t$, and consequently, being negative at $t=1$, it is negative until some point $t=\tau_1$, where it equals zero and then is strictly positive. And this point $\tau_1$ exists, because
$\frac{\partial \mathsf{V}(1, t)}{\partial H}\to+\infty$ as $t\to\infty$, according to Lemma \ref{lemlem1}.  It is clear that  for any fixed $ t\in(0, \tau_1]$ function   $\mathsf{V}(H, t)$ decreases in $H\in[1/2, 1]$, and we get item (a). 

Now, consider
\begin{equation}\label{derV2}
   \frac{\partial \mathsf{V}(1/2, t)}{\partial H}
   = 2e^t t \log{t}
   +\int_0^t\left(e^{2t-z}-e^{z}\right)z \log{z}\, dz.
\end{equation}
Again, $\frac{\partial \mathsf{V}(1/2, t)}{\partial H}<0$ for $0<t\le 1$, 
\begin{gather*}
   \frac{\partial^2 \mathsf{V}(1/2, t)}{\partial H \partial t}
   = 2e^t t \log{t} +2e^t \log{t} + 2e^t
   +2e^{2t}\int_0^t e^{-z} z \log{z}\, dz.
\end{gather*}
Function $z\log z$ at the interval $[0,1$ has minimal value at point $z=e^{-1}$, and this minimal value equals $-e^{-1}$.
Therefore, at point $t=1$  
\begin{gather*}
   \frac{\partial^2 \mathsf{V}(1/2, 1)}{\partial H \partial t}
   = 2e\left(1+ e \int_0^1 e^{-z} z \log{z}\, dz
   \right)
   >2e\left(1+ e(-e^{-1}) \int_0^1 e^{-z}\, dz
   \right)>0,
\end{gather*}
and
\begin{gather*}
   \frac{\partial^2 \mathsf{V}(1/2, t)}{\partial H \partial t}
   = 2e^{2t}
   \left(e^{-t}t \log{t} +e^{-t} \log{t} + e^{-t}
   +\int_0^t e^{-z} z \log{z}\, dz
   \right),
\end{gather*}
where the value in the brackets is positive at point $t=1$, and increases in $t$, because its derivative equals
\begin{gather*}
    -e^{-t}t \log{t}+e^{-t} \log{t} + e^{-t}
    -e^{-t} \log{t}+ e^{-t} \frac{1}{t} -e^{-t}
    +e^{-t}t \log{t}=\frac{e^{-t}}{t}>0.
\end{gather*}
Again, $\lim\limits_{t\to\infty}\frac{\partial \mathsf{V}(1/2, t)}{\partial H}=+\infty$, according to Lemma \ref{lemlem1}. It means that there exists a point $1<\tau_{1/2}$ such that $\frac{\partial \mathsf{V}(1/2, t)}{\partial H}$ is negative for $0<t<\tau_{1/2}$, equals zero at $t=\tau_{1/2}$ and positive for $t>\tau_{1/2}$.
Also, $\frac{\partial \mathsf{V}(1/2, t)}{\partial H}<\frac{\partial \mathsf{V}(1, t)}{\partial H}$, therefore $\tau_1<\tau_{1/2}$.

Now, how to compare points $\tau_1$ and $\tau_{1/2}$ with  $\log 3$?   According to Lemma~\ref{lem:V1/2-V1}, $\mathsf{V}(1/2, \log 3)=\mathsf{V}(1, \log 3).$ However, we know that for all $t\le \tau_1$ $\mathsf{V}(H, t)$ decreases in $H$ whence for all $t\le \tau_1$ $\mathsf{V}(1/2, t)>\mathsf{V}(1, t)$ and consequently, $\tau_1< \log{3}$. Similarly, after point $\tau_{1/2}$ 
  $\mathsf{V}(H, t)$ increases in $H$, whence for all $t\ge \tau_{1/2}$ $\mathsf{V}(1/2, t)<\mathsf{V}(1, t)$ and consequently, $\tau_{1/2}> \log{3}$. Additionally, for any $ t\in (\tau_1,  \log 3)$ we have that $\frac{\partial \mathsf{V}(1/2, t)}{\partial H}<0$ and   $\frac{\partial \mathsf{V}(1, t)}{\partial H}>0$, therefore there exists a value of Hurst index $1/2<H_t<1$  such that $\mathsf{V}(H, t)$ decreases in
    $H\in[1/2, H_t]$ and increases in $H\in[H_t, 1]$, and
    $\mathsf{V}(1, t)<\mathsf{V}(1/2, t)$. Moreover, according to item $(iii)$ of Lemma \ref{lem:V1/2-V1}, $\mathsf{V}(1/2, t)>\mathsf{V}(1, t)$ for $0< t< \log 3$, and we get item (b). Item (c) is evident now, and items (d) and (e) are established, taking in account all previous calculations, similarly to (b).  Hence, the proof follows. 
\end{proof}
\begin{figure}[h]
\centering
\includegraphics[scale=0.8]{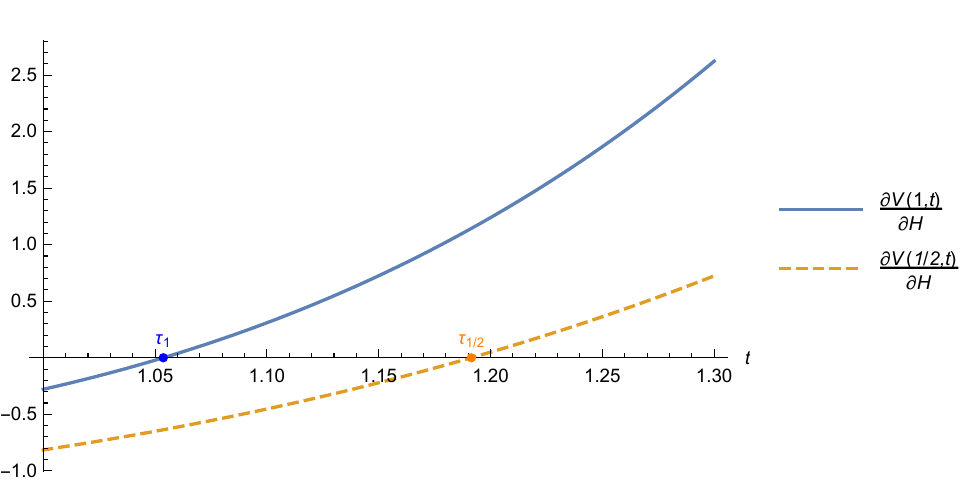}
\caption{Graphs of $\frac{\partial \mathsf{V}(1,H)}{\partial H}$ and  $\frac{\partial \mathsf{V}(1/2,H)}{\partial H}$\label{fig:tau}}
\end{figure}

\begin{remark}
The graphs of the partial derivatives $\frac{\partial \mathsf{V}(1,H)}{\partial H}$ and $\frac{\partial \mathsf{V}(1/2,H)}{\partial H}$, as defined in \eqref{derV1} and \eqref{derV2}, respectively, are shown in Figure~\ref{fig:tau}. By solving the equations $\frac{\partial \mathsf{V}(1,H)}{\partial H} = 0$ and $\frac{\partial \mathsf{V}(1/2,H)}{\partial H} = 0$ numerically, we determine that the approximate solutions are
\[
\tau_1  \approx 1.05368
\quad\text{and}\quad
\tau_{1/2} \approx 1.19142.
\]
It is worth noting that $\log 3 \approx 1.09861$, which implies that the inequalities
$\tau_{1/2}>\log 3 >\tau_{1}>1$
stated in Theorem~\ref{IBYMtheorem_2} are satisfied.
\end{remark}

\begin{figure}[h]
\centering
\subcaptionbox{$t=1$\label{fig:1a}}{\includegraphics[scale=0.55]{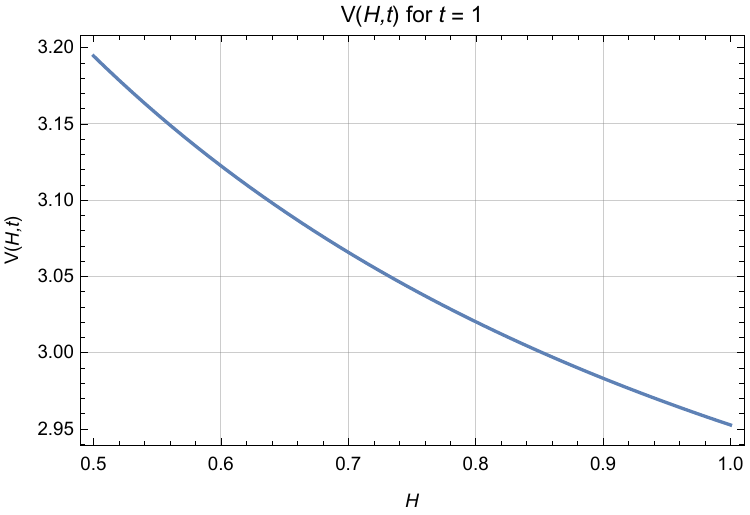}}
\qquad
\subcaptionbox{$t=1.08$ (minimum at $H_{1.08}\approx 0.8102$)\label{fig:1b}}{\includegraphics[scale=0.55]{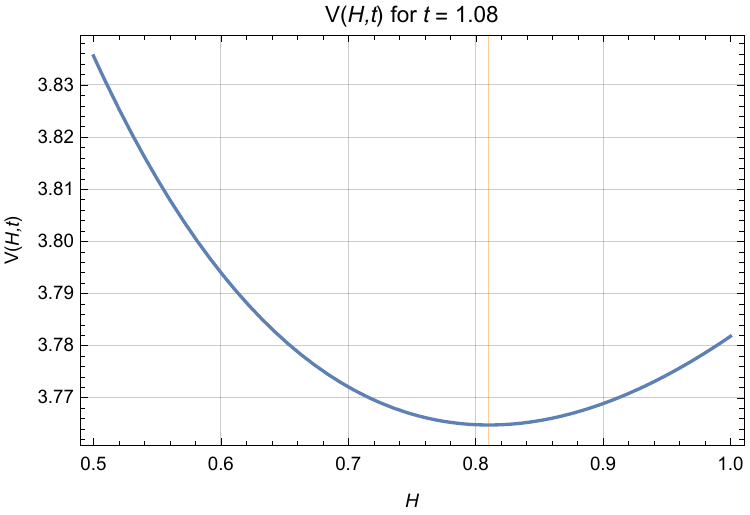}}
\\[12pt]
\subcaptionbox{$t = \log 3$ (minimum at $H_{\log3}\approx 0.7232$)\label{fig:1e}}{\includegraphics[scale=0.55]{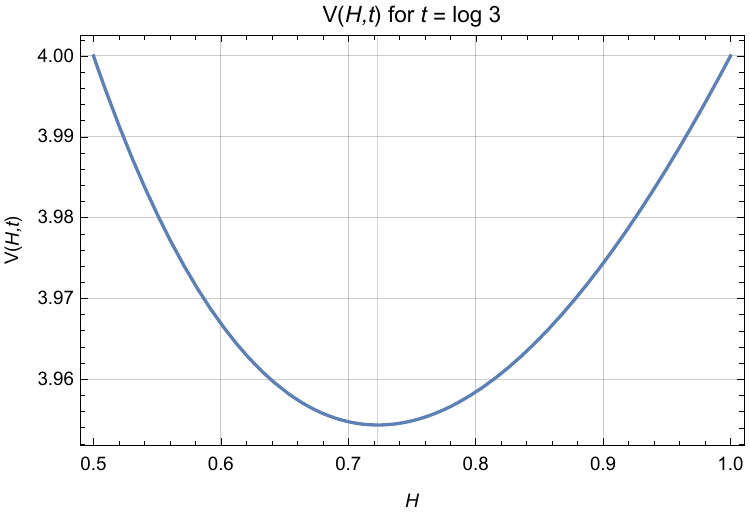}}
\qquad
\subcaptionbox{$t=1.15$ (minimum at $H_{1.15}\approx 0.5731$)\label{fig:1c}}{\includegraphics[scale=0.55]{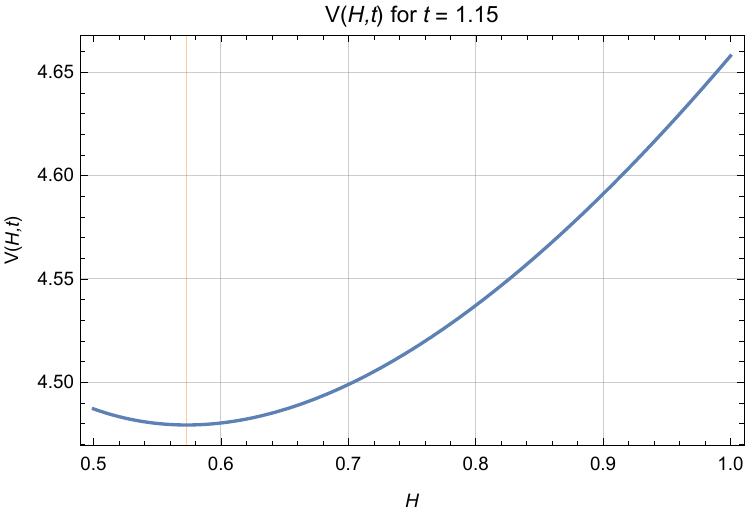}}
\\[12pt]
\subcaptionbox{$t=1.25$\label{fig:1d}}{\includegraphics[scale=0.55]{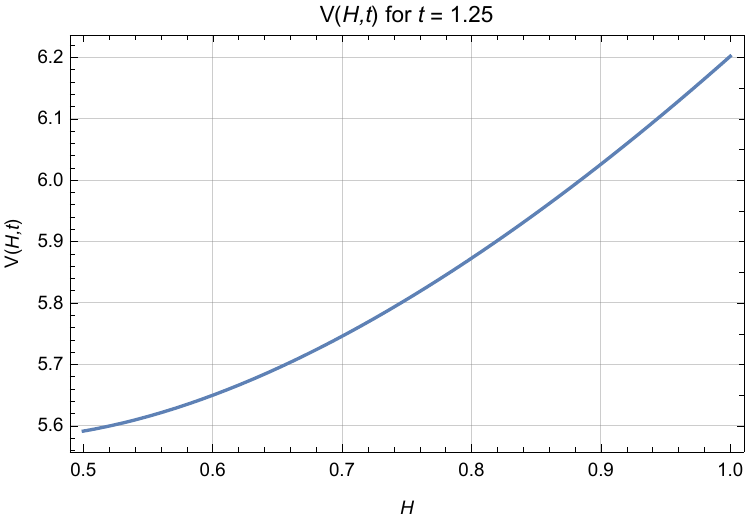}}
 \caption{Graphs of $\mathsf{V}(H,t)$ as functions of $H\in(\frac12,1)$}\label{fig:1}
\end{figure}
\begin{remark}
Figure~\ref{fig:1} presents the results of Theorem~\ref{IBYMtheorem_2}.
Figures~\ref{fig:1a}--\ref{fig:1d} display the graphs of the function $\mathsf{V}(H, t)$ at specific points $t = 1, 1.08, \log 3, 1.15$, and $1.25$, corresponding to various cases described in the theorem. The observations are as follows:
\begin{itemize}
\item For $t=1 \in (0,\tau_1)$, the function $\mathsf{V}(H,1)$ decreases w.r.t.\ $H$ in the interval $H\in[\frac12,1]$ (Figure~\ref{fig:1a}).
\item For $t=1.08 \in (\tau_1,\log 3)$, the function $\mathsf{V}(H,1.08)$ decreases as 
$H$ increases from $\frac12$ to $H_{1.08}\approx  0.8102$, and subsequently increases. Additionally, $\mathsf{V}(1, 1.08)<\mathsf{V}(1/2, 1.08)$ (Figure~\ref{fig:1b}).
    \item for $t= \log 3$, the function $\mathsf{V}(\log 3, H)$ decreases as $H$ increases from $\frac12$ to $H_{\log 3}\approx  0.7232$, and then increases. Moreover, $\mathsf{V}(1/2, \log 3)=\mathsf{V}(1, \log 3)$ (Figure~\ref{fig:1e}).
    \item For $t = 1.15 \in (\log3, \tau_{1/2})$, the function $\mathsf{V}(H, 1.15)$ decreases as $H$ increases from $\frac12$ to $H_{1.15}\approx 0.5731$, and then increases. In this case, $\mathsf{V}(1/2, 1.15) < \mathsf{V}(1, 1.15)$ (Figure~\ref{fig:1c}).
    \item for $t = 1.25 >\tau_{1/2}$, the function $\mathsf{V}(H, 1.25)$ increases w.r.t.\ $H$ in the interval $H\in[\frac12,1]$ (Figure~\ref{fig:1d}).
\end{itemize}
\end{remark}

\subsection{Entropy of the  EWIFG-process  as the function of \texorpdfstring{$H\in[0,1/2]$}{H from [0,1/2]}}
Now,  we continue  with the second  main result that characterizes   the behavior of the variance $\mathsf{V}(H, t)=\mathsf{E} (X_t^H)^2,\ H\in[0, 1/2]$ as the function of $H$ for different $t>0$. So, for any $ H\in[0, 1/2]$ consider the variance $\mathsf{V}(H, t)$ of the process $X^H$ defined by \eqref{varnevar1}. The situation at boundary points $0$ and $\frac12$ is now simpler than in the previous case, because for any $t>0$ $$\mathsf{V}(0, t)=\frac{e^{2t}+1}{2}>\frac{e^{2t}-1}{2}=\mathsf{V}(1/2, t).$$
The first derivative of the variance $\mathsf{V}(H, t)$ w.r.t. $H$, for any $t>0$  can be calculated as in \eqref{deriv} and also presented in the form 
$$
\frac{\partial \mathsf{V}(H, t)}{\partial H}
   = 2e^t t^{2H}\log{t}
   +e^{2t}\int_0^t  e^{-z}z^{2H}\log{z}\,dz
   -\int_0^t  e^{z}z^{2H}\log{z}\,dz.
$$
In particular,   
 \begin{gather*}
   \frac{\partial \mathsf{V}(0, t)}{\partial H}
   =2e^t \log{t}
   +e^{2t}\int_0^t  e^{-z}\log{z}\,dz
   -\int_0^t  e^{z}\log{z}\,dz,\\
  \frac{\partial \mathsf{V}(1/2, t)}{\partial H}
   =2e^t t\log{t}
   +e^{2t}\int_0^t  e^{-z} z\log{z}\,dz
   -\int_0^t  e^{z}z\log{z}\,dz. 
\end{gather*}
and $\frac{\partial \mathsf{V}(0, 0)}{\partial H}$ does not exist. Furthermore, according to \eqref{secder} that is also true for any $0<H<1/2$ and $t>0$, the second derivative of the variance $\mathsf{V}(H, t)$ with respect to $H$ is strictly positive: $\frac{\partial^2 \mathsf{V}(H, t)}{\partial^2 H} >0$, and consequently, $\frac{\partial \mathsf{V}(H, t)}{\partial H}$ strictly increases in $H\in(0, 1/2]$. Taking this into account, let us formulate an auxiliary result. 
\begin{remark}\label{rem-compar0-1}
 Inequality $\frac{e^{2t}+1}{2}>(e^t-1)^2$ hols for $t<\log(2+\sqrt{3})$. So, $\mathsf{V}(0, t)>\mathsf{V}(1, t)$ for $t<\log(2+\sqrt{3})$ and $\mathsf{V}(0, t)<\mathsf{V}(1, t)$ for $t>\log(2+\sqrt{3})$.
\end{remark}
In what follows, $\tau_{1/2}>1$ is the point introduced in Theorem \ref{IBYMtheorem_2}
\begin{lemma}\label{lemlem_aux}
\begin{enumerate}[(i)]
\item For all $t>0$ $\frac{\partial \mathsf{V}(0, t)}{\partial H}<0$.
\item For any fixed $ t\in(0, \tau_{1/2})$
$\frac{\partial \mathsf{V}(H, t)}{\partial H}<0$
for any $H\in[0, 1/2]$.
\item For $t=\tau_{1/2}$
$\frac{\partial \mathsf{V}(0, t)}{\partial H}<0$,
$\frac{\partial \mathsf{V}(H, t)}{\partial H}$ increases in $H$ and $\frac{\partial \mathsf{V}(1/2, t)}{\partial H}=0$.
\item For any fixed $t>\tau_{1/2}$
$\frac{\partial \mathsf{V}(0, t)}{\partial H}<0$,
$\frac{\partial \mathsf{V}(H, t)}{\partial H}$ increases in $H$, equals zero at the unique point $\widehat{H}_t$ and is strictly positive for $H\in(\widehat{H}_t, 1/2]$.
\end{enumerate}
\end{lemma}
\begin{proof} $(i)$
Recall that for $t>0$
\begin{align*}
    \psi(t)&\coloneqq\frac{\partial \mathsf{V}(0, t)}{\partial H}=2e^t \log{t}
   +e^{2t}\int_0^t  e^{-z}\log{z}\,dz
   -\int_0^t  e^{z}\log{z}\,dz\\
   &=2e^t \log{t}+\int_0^t  (e^{2t-z}-e^{z})\log{z}\,dz
\end{align*}
The latter equality implies that for $0<t\le 1$
$\frac{\partial \mathsf{V}(0, t)}{\partial H}<0$. So, it is sufficient to consider $t>1$.

It follows from Lemma \ref{lemlem1} that $\int_0^{t}e^{-z}\log{z}\,dz$ being increasing in $t>1$, is strictly negative. Therefore,
$\lim_{t\to\infty}\psi(t)=-\infty$.
Assume that $\psi(t)$ achieves its maximum at some point $t_0$. Then  
\begin{gather*}
    \psi'(t_0)=2e^{t_0}\log{t_0}+\frac{2e^{t_0}}{t_0}
    +2e^{2t_0}\int_0^{t_0}e^{-z}\log{z}\,dz=0,
\end{gather*}
whence
$$
e^{-t_0}\log{t_0}+\frac{e^{-t_0}}{t_0}
+\int_0^{t_0}e^{-z}\log{z}\,dz=0.
$$

Consider the function
\begin{gather*}
    \varphi(t)=e^{-t}\log{t}+\frac{e^{-t}}{t}
+\int_0^{t}e^{-z}\log{z}\,dz,\ t\ge 1.
\end{gather*}
It has the properties:
\begin{gather*}
    \varphi(1)=e^{-1}+\int_0^{1}e^{-z}\log{z}\,dz
    <e^{-1}-e^{-1}\int_0^{1}\log{z}\,dz=0,
\end{gather*}
and its derivative equals
\begin{gather*}
   \varphi'(t)= -e^{-t}\log{t}+\frac{e^{-t}}{t}
   -\frac{e^{-t}}{t}-\frac{e^{-t}}{t^2}+e^{-t}\log{t}
   =-\frac{e^{-t}}{t^2}<0,\ t\ge 1.
\end{gather*}
Therefore $ \varphi(t)<0, \,t\ge 1$, and point $t_0$ does not exist. So, $\frac{\partial \mathsf{V}(0, t)}{\partial H}<0$, and we get $(i)$.

Let us prove $(ii)$--$(iv)$. Obviously, since the derivative strictly increases, it implies that $\frac{\partial \mathsf{V}(H, t)}{\partial H}<0$ for all $H\in[0, 1/2]$ if (and only if)  $\frac{\partial \mathsf{V}(1/2, t)}{\partial H}\le 0$.  However, the behavior of the derivative $\frac{\partial \mathsf{V}(1/2, t)}{\partial H} $ was investigated in the proof of Theorem \ref{IBYMtheorem_2}, and it was stated  that there exists a point $1<\tau_{1/2}$ such that $\frac{\partial \mathsf{V}(1/2, t)}{\partial H}$ is negative for $0<t<\tau_{1/2}$, equals zero at $t=\tau_{1/2}$ and positive for $t>\tau_{1/2}$. From this we immediately get   $(ii)$--$(iv)$. 
\end{proof}
The next result is an immediate consequence of Lemma \ref{lemlem_aux}. Point $\tau_{1/2}>1$, as we said before,  was introduced in Theorem \ref{IBYMtheorem_2}. 
\begin{theorem}\label{IBYMsimsim-3} The behavior of $\mathsf{V}(H, t)=\mathsf{E} (X_t^H)^2,\ H\in[0, 1/2]$ is the following:
 $  \mathsf{V}(H, t) $ strictly decreases in $H\in[0,1/2]$ for $t\le \tau_{1/2}$, and for any $t> \tau_{1/2}$ there exists such   unique point $\widehat{H}_t\in(0,1/2)$  that   $\mathsf{V}(H, t)$ decreases for $H\in[0, \widehat{H}_t]$ and increases for $H\in(\widehat{H}_t, 1/2]$.
\end{theorem}

\begin{figure}[thb]
\centering
\subcaptionbox{$t=1.15$\label{fig:2c}}{\includegraphics[scale=0.55]{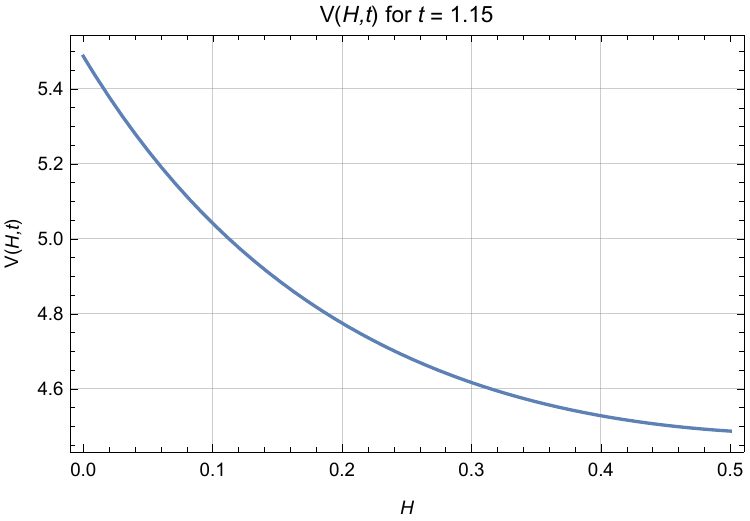}}
\qquad
\subcaptionbox{$t=1.25$ (minimum at $\widehat H_{1.25}\approx 0.4314$)\label{fig:2d}}{\includegraphics[scale=0.55]{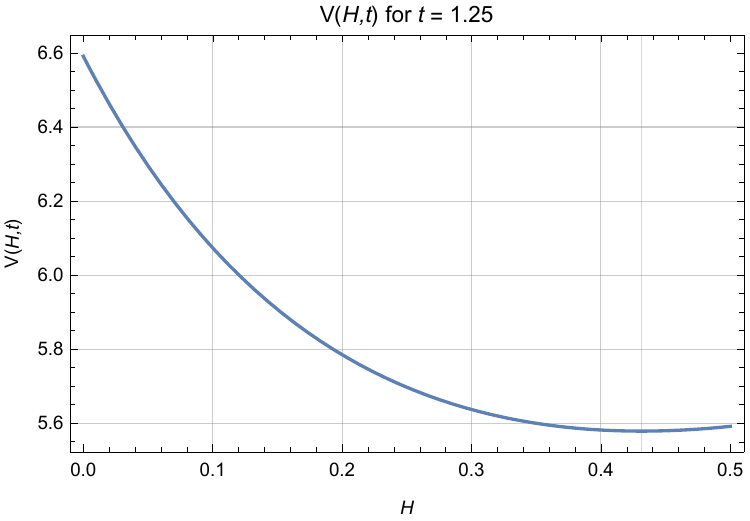}}
 \caption{Graphs of $\mathsf{V}(H,t)$ as functions of $H\in(0,\frac12)$}\label{fig:2}
\end{figure}

\begin{remark}
Figure \ref{fig:2} illustrates the behavior of the function 
$\mathsf{V}(H,t)$ for 
$H\in[0,\frac12]$ at $t = 1.15$ and $t = 1.25$. This supports the result of Theorem \ref{IBYMsimsim-3}, as detailed below:
\begin{itemize}
\item
For $t = 1.15 < \tau_{1/2}$, the function $\mathsf{V}(H,1.15)$ strictly decreases w.r.t.\ $H$ in the interval $H\in[0,\frac12]$ (Figure \ref{fig:2c}).
\item for $t = 1.25 >\tau_{1/2}$, the function $\mathsf{V}(H, 1.25)$ decreases as $H$ increases from $\frac12$ to $\widehat H_{1.25}\approx 0.4314$, and then increases (Figure \ref{fig:2d}).
\end{itemize}
\end{remark}

\begin{corollary}\label{IBYMcorcor-2} \begin{itemize}
    \item [(i)] Let $t\in(0, \tau_1]. $ Then $\mathsf{V}(H, t)$ decreases for $H\in[0, 1]$.
  \item [(ii)] Let $t\in(\tau_1, \tau_{1/2}]$.   Then $\mathsf{V}(H, t)$ decreases for $H\in[0, H_t]$ and increases for  $H\in[H_t, 1]$,   $H_t>1/2$.  
   \item [(iii)]  Let $t>\tau_{1/2}$.   Then $\mathsf{V}(H, t)$ decreases for $H\in[0, \widehat{H}_t]$ and increases for  $H\in[\widehat{H}_t, 1]$,   $\widehat{H}_t<1/2$.
\end{itemize}    
\end{corollary}

\begin{figure}%[ht]
\centering
\subcaptionbox{$t=1$\label{fig:3a}}{\includegraphics[scale=0.55]{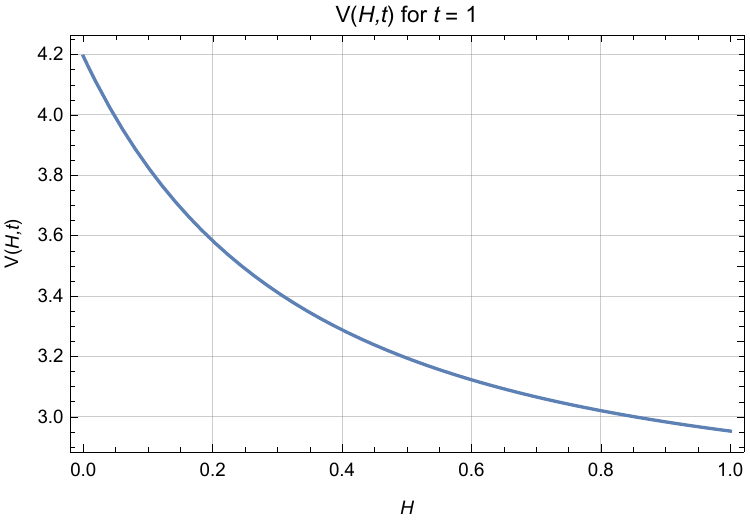}}
\qquad
\subcaptionbox{$t=1.08$ (minimum at $H_{1.08}\approx 0.8102$)\label{fig:3b}}{\includegraphics[scale=0.55]{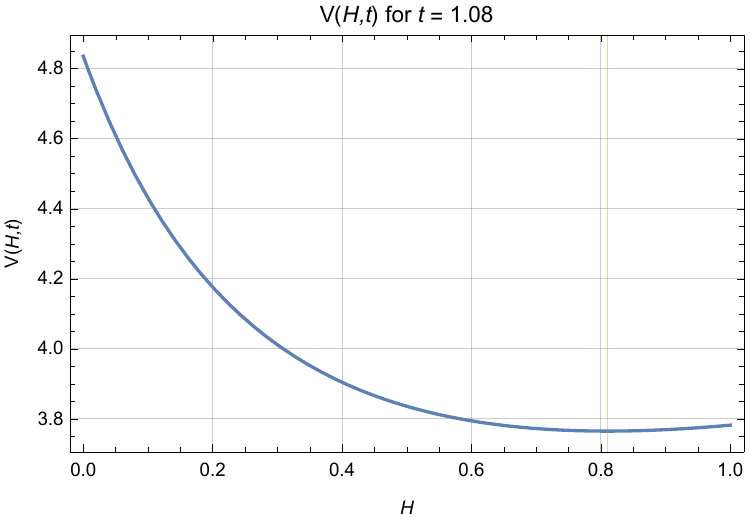}}
\\[12pt]
\subcaptionbox{$t=1.15$ (minimum at $H_{1.15}\approx 0.5731$)\label{fig:3c}}{\includegraphics[scale=0.55]{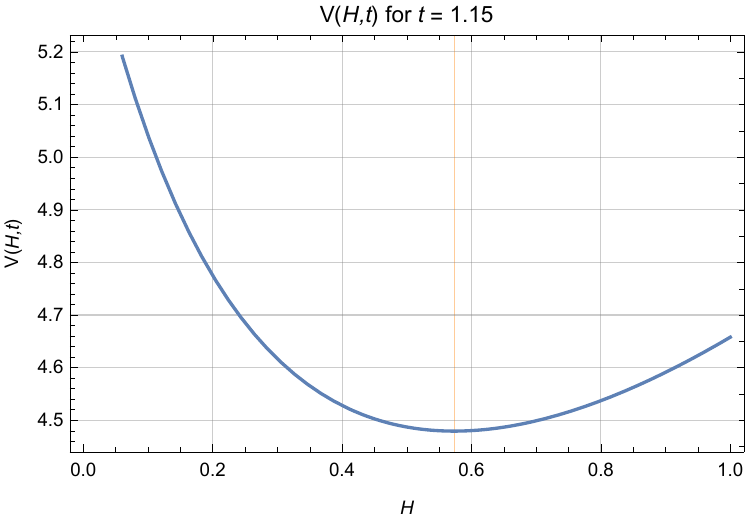}}
\qquad
\subcaptionbox{$t=1.25$ (minimum at $\widehat H_{1.25}\approx 0.4314$)\label{fig:3d}}{\includegraphics[scale=0.55]{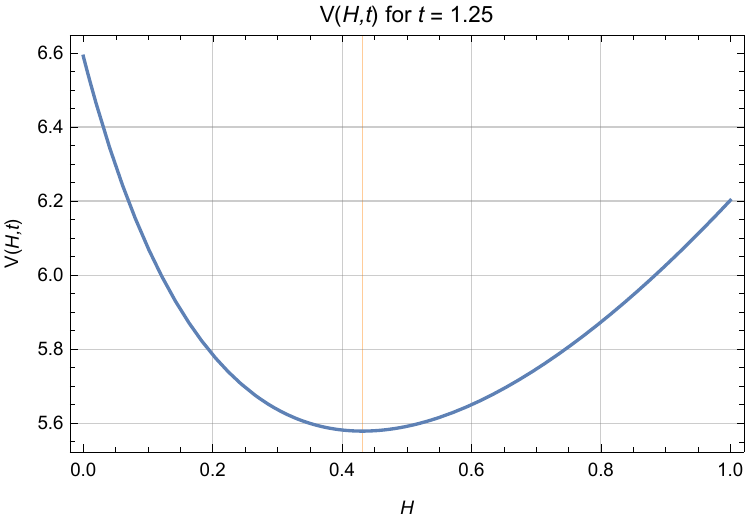}}
\\[12pt]
\subcaptionbox{$t= \log(2+\sqrt3)$ (minimum at $\widehat H_{\log(2+\sqrt3)}\approx 0.3801$)\label{fig:3g}}{\includegraphics[scale=0.55]{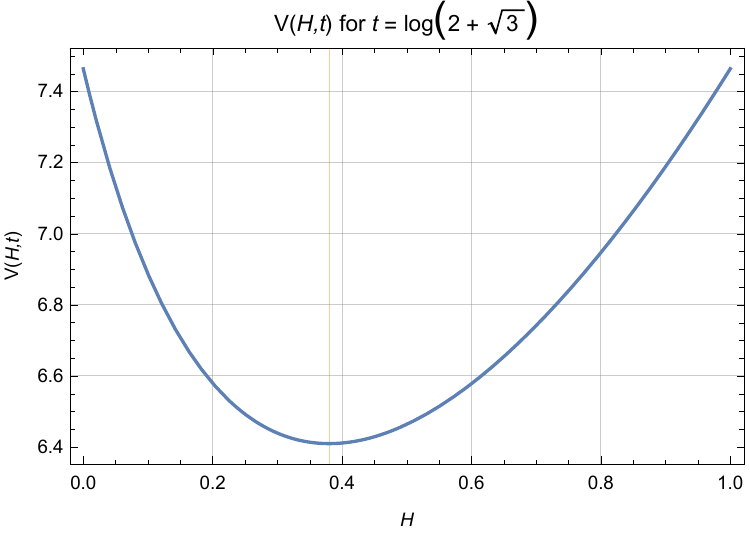}}
\qquad
\subcaptionbox{$t=1.5$ (minimum at $\widehat H_{1.5}\approx 0.3055$)\label{fig:3f}}{\includegraphics[scale=0.55]{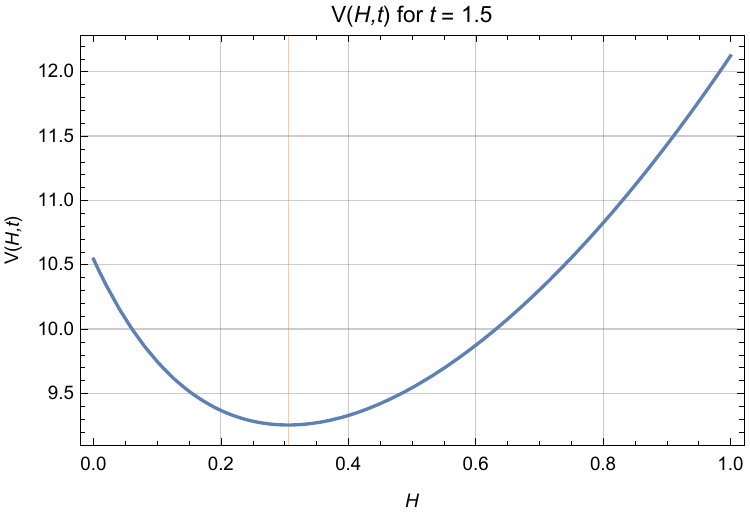}}
 \caption{Graphs of $\mathsf{V}(H,t)$ as functions of $H\in(0,1)$}\label{fig:3}
\end{figure}

\begin{remark}
Figure~\ref{fig:3} presents the graphs of the variance $\mathsf{V}(H,t)$ as a function of $H$ for all $H \in [0,1]$, evaluated at the points 
$t = 1, 1.08, 1.15, 1.25$, and $1.5$.

These graphs support the statement in Corollary~\ref{IBYMcorcor-2} and are consistent with Remark~\ref{rem-compar0-1}. 
Specifically, we observe the following behavior of $\mathsf{V}(H,t)$:
\begin{itemize}
\item
If $t<\tau_1$, then the function 
$\mathsf{V}(H,t)$ decreases monotonically with respect to $H$ over the interval 
$H\in[0,1]$ (Figure~\ref{fig:3a}).
\item
If  $t>\tau_1$, then the function $\mathsf{V}(H,t)$ first decreases and then increases, exhibiting a single minimum. The position of this minimum depends on the value of $t$:
\begin{itemize}
\item
For $t\in(\tau_1,\tau_{1/2})$, the minimum lies within the interval $(\frac12,1)$ (Figures~\ref{fig:3b}--\ref{fig:3c}).
\item
For $t>\tau_{1/2}$, the minimum lies within the interval $(0,\frac12)$ (Figures~\ref{fig:3d}--\ref{fig:3f}).
\end{itemize}
\end{itemize}
Moreover, we note the following relationships:
\begin{itemize} 
\item 
$\mathsf{V}(0,t) > \mathsf{V}(1,t)$ for $t < \log(2 + \sqrt3)$  (Figures~\ref{fig:3a}--\ref{fig:3d}), 
\item 
$\mathsf{V}(0,t) = \mathsf{V}(1,t)$ at $t = \log(2 + \sqrt3)$  (Figure~\ref{fig:3g}), and
\item 
$\mathsf{V}(0,t) < \mathsf{V}(1,t)$ for $t = 1.5 > \log(2 + \sqrt3) \approx 1.31696$ (Figure~\ref{fig:3f}).
\end{itemize}

The surface plot of the variance $\mathsf{V}(H,t)$ as a function of $(H,t)$ over the region $0\le H\le 1$, $0\le t\le 1.5$ is shown in Figure~\ref{fig:4}.
\end{remark}

\begin{figure}
\centering
\includegraphics[scale=0.7]{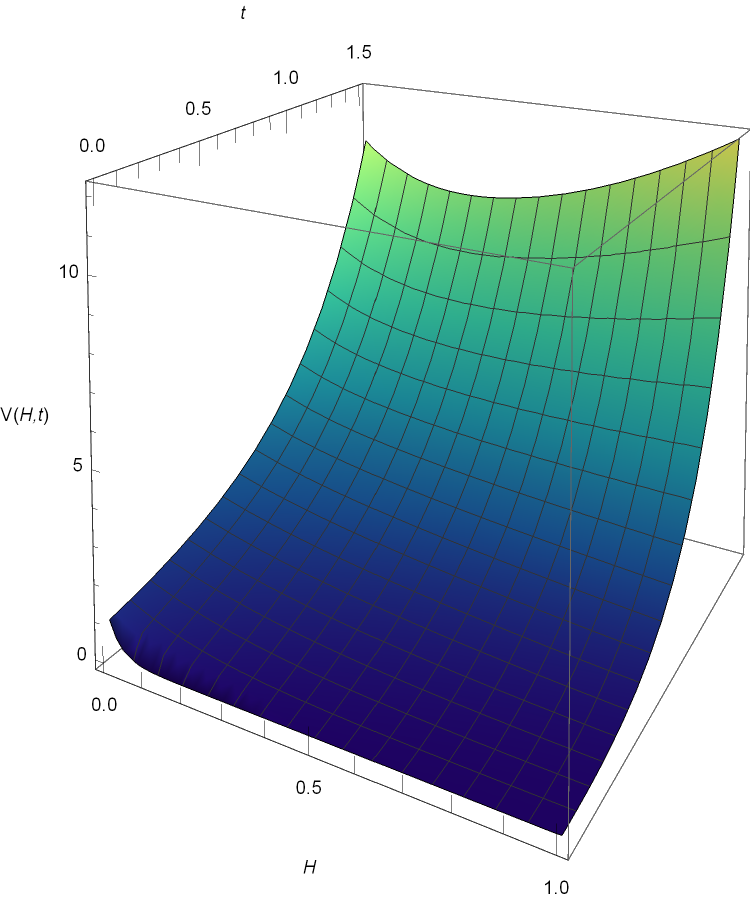}
\caption{Surface plot of $\mathsf{V}(H,t)$ as a function of $(H,t)$\label{fig:4}}
\end{figure}

\begin{remark} Behavior of entropy of EWIFG-process $X^H$ that mimics the behavior of variance, differs substantially from the behavior of the entropy of fractional Gaussian noise, described in \cite{MishuraRalchenkoSchilling2022}. The latter one increases in $H\in[0,1/2]$ and decreases in $H\in[1/2,1]$. So, exponent $e^t$ changes the behavior of fractional process crucially.     
\end{remark}

\subsection{Monotonicity and asymptotic behavior of variance of EWIFG-process in time}
Note that variance $\mathsf{V}(H, t)$ increases in $t$ because
$$\frac{\partial \mathsf{V}(H, t)}{\partial t}=e^tt^{2H}+2He^tt^{2H-1}+e^{2t}\int_0^te^{-z}z^{2H}dz>0,\quad t\ge 0.$$    Consider the asymptotic behavior as $t\to \infty$ of $\mathsf{V}(H, t)=\mathsf{E}\left(X_t^H\right)^2$ for any
$H\in[0, 1]$.

\begin{lemma}
    For any $H\in[0, 1]$   
    \begin{gather}\label{asympt}
        \lim_{t\to\infty}e^{-2t}\mathsf{V}(H, t) =\Gamma(2H+1)/2.
    \end{gather}
\end{lemma}
\begin{proof}First, consider the boundary values. 
Let $H=0$. Then $\mathsf{V}(0, t)=\frac{e^{2t}+1}{2}$, whence the limit \eqref{asympt} equals $ 1/2=\Gamma(1)/2$. For $H=1/2$ the limit \eqref{asympt} equals $ 1/2=\Gamma(2)/2$.
 Let $H=1$. Then the limit \eqref{asympt} equals $ 1=\Gamma(3)/2$.

Now, let $H\in(1/2, 1)$. Then
$$
\mathsf{V}(H,t)
=2H(2H-1)\int_0^te^v\int_0^v e^u(v-u)^{2H-2}dudv,
$$
and applying L'H\^{o}pital's rule twice, we get that
\begin{align*}
\lim_{t\to\infty}\frac{\mathsf{V}(H,t)}{e^{2t}}
&=2H(2H-1)\lim_{t\to\infty}
\frac{\int_0^te^u(t-u)^{2H-2}du}{e^{t}}\\
&=H(2H-1)\lim_{t\to\infty}\int_0^te^{-z}z^{2H-2}dz
=H(2H-1)\Gamma(2H-1)\\
&=H\Gamma(2H)=\frac12 \Gamma(2H+1).
\end{align*}

Let $H\in(0,1/2)$. Then, by \eqref{varnevar1}, we obtain
\begin{align*}
\lim_{t\to\infty}\frac{\mathsf{V}(H,t)}{e^{2t}}
&= \lim_{t\to\infty}\left(\frac{t^{2H}}{e^t}
+\frac12 \int_0^t e^{-z}z^{2H}dz
-\frac12 e^{-2t}\int_0^t e^{z}z^{2H}dz \right)
\\
&=\frac12 \Gamma(2H+1),
\end{align*}
and we get \eqref{asympt} for any $H\in[0,1]$.
\end{proof}

\section{EWIFG-process with exponent \texorpdfstring{$e^{kt}$}{Exp(kt)}}\label{IBYMsecsec4}
Let us consider the process $X^{H,k}=\{X^{H,k}_t,\ t\ge 0\}$ of the form 
$$
X^{H,k}_t=\int_0^te^{ks}dB_s^H,\quad k\in\mathbb{R}\setminus \{0\},\quad H\in[0, 1].
$$
Its square characteristics can be calculated according to Lemma \ref{lem:covar}, replacing everywhere exponent $e^s$ for $e^{ks}$ for any argument $s\ge 0.$ Also, $X^{H,k}$ can be extended to $H=0$ and $H=1$ similarly to Subsections \ref{extendH=1} and \ref{extendH=0},  and we obtain the following variances:

\begin{gather}\label{eq:varsk0}
\mathsf{V}(0, k, t)= \begin{cases}
\frac{e^{2kt}+1}{2}, & t>0,\\
0, & t=0,
\end{cases}\qquad  \mathsf{V}(1/2, k, t)=\int_0^te^{2ks}ds=\frac{e^{2kt}-1}{2k}, 
\end{gather}
and 
\begin{gather}\label{eq:varsk1}\mathsf{V}(1, k, t)=\left(\int_0^te^{ks}ds\right)^2
=\frac{e^{2kt}-2e^{kt}+1}{k^2}.
\end{gather}
Let us  compare the variances for $H=1/2$ and $H=1$ respectively.
\begin{lemma}
Variances $\mathsf{V}(1/2, k, t)$ and $\mathsf{V}(1, k, t)$ have the following properties:
\begin{enumerate}[1)]
    \item $\mathsf{V}(1/2, k, t)=\mathsf{V}(1, k, t)=0,\ t=0,$
    \item $\mathsf{V}(1/2, k, t)=\mathsf{V}(1, k, t)
    =\frac{4}{(2-k)^2},
    \ t=\frac{1}{k}\log\frac{k+2}{2-k},\ |k|<2,$
   
    \item\label{eq:VHkt2} $\mathsf{V}(1/2, k, t)>\mathsf{V}(1, k, t),\
 0 < t <\frac{1}{k}\log\frac{k+2}{2-k},\ |k|<2,$
    \item\label{eq:VHkt3} $\mathsf{V}(1/2, k, t)<\mathsf{V}(1, k, t),\
 t >\frac{1}{k}\log\frac{k+2}{2-k},\ |k|<2,$
  \item\label{eq:VHkt1} $\mathsf{V}(1/2, k, t)>\mathsf{V}(1, k, t),\ t>0,\ |k|\ge 2.$
\end{enumerate}
\end{lemma}
\begin{proof} The first two equalities are immediate. Just note that   $\frac{1}{k}\log\frac{k+2}{2-k}>0$ for $|k|<2$.
Now, let us compare the values of $\mathsf{V}(1/2, k, t)$ and $\mathsf{V}(1, k, t)$ for arbitrary $t>0$.  First, let us find such $t>0$ and $k\ne 0$ that
 
$$
\frac{e^{2kt}-1}{2k} \ge \frac{e^{2kt} - 2e^{kt} + 1}{k^2}.
$$

Multiplying both sides of this inequality by $2k^2>0$ leads to the following relation:
$$
(k-2)e^{2kt} + 4e^{kt} - (k+2) \ge 0,
$$
which can be represented in the form
\begin{gather}\label{eq:varskge}
 (e^{kt} - 1)\left((k-2)e^{kt} + (k+2)\right) \ge 0.   
\end{gather}
Let $|k|<2$. Then $(k - 2) < 0$, therefore, dividing both sides of \eqref{eq:varskge} by $(k - 2)$ give us the equivalent relation
$$
(e^{kt} - 1)\left(e^{kt} - \frac{k+2}{2-k}\right) \le 0.
$$

In the case $0<k<2$, we obtain
\begin{gather*}
  (e^{kt} - 1) > 0 \mbox{ and }
  kt < \log{\frac{k+2}{2-k}} \mbox{ for } 0<t<\frac{1}{k}\log\frac{k+2}{2-k},
\end{gather*}
and, for $-2<k<0$,
\begin{gather*}
  (e^{kt} - 1) < 0 \mbox{ and }
  kt > \log{\frac{k+2}{2-k}} \mbox{ for } 0<t<\frac{1}{k}\log\frac{k+2}{2-k},
\end{gather*}
consequently, \ref{eq:VHkt2} and \ref{eq:VHkt3} are proved.
 
Finally, for any $t>0$ we have
\begin{gather*}
  (e^{kt} - 1) > 0 \mbox{ and }
  (k-2)e^{kt} + (k+2)\ge 4>0 \mbox{ for } k \ge 2,\\
  (e^{kt} - 1) < 0 \mbox{ and }
  (k-2)e^{kt} + (k+2)<0 \mbox{ for } k \le -2,
\end{gather*}
thus, relation in \ref{eq:VHkt1} follows.
\end{proof}

\begin{lemma}
Variances $\mathsf{V}(0, k, t)$ and $\mathsf{V}(1/2, k, t), k\neq 0,$ have the following properties:
\begin{enumerate}[1),left=1em]
\item \label{zeroonehalf1} $\mathsf{V}(0, k, t)=\mathsf{V}(1/2, k, t)=0$, $t=0$, $k\in\mathbb{R}\setminus\{0\}$,
\item \label{zeroonehalf2}
$\mathsf{V}(0, k, t)=\mathsf{V}(1/2, k, t)=\frac{1}{1-k}$,
$t=\frac{1}{2k}\log\left(\frac{k+1}{1-k}\right)$, $|k|<1$,
\item \label{zeroonehalf3}$\mathsf{V}(0, k, t)>\mathsf{V}(1/2, k, t)$, $t<\frac{1}{2k}\log\left(\frac{k+1}{1-k}\right)$, and $\mathsf{V}(0, k, t)<\mathsf{V}(1/2, k, t)$, $t>\frac{1}{2k}\log\left(\frac{k+1}{1-k}\right)$, $|k|<1$,
\item \label{zeroonehalf5} $\mathsf{V}(0, k, t)>\mathsf{V}(1/2, k, t)$, $t>0$, $|k|\ge 1$.
\end{enumerate}
\end{lemma}
\begin{proof}
Equalities \ref{zeroonehalf1} and \ref{zeroonehalf2} are evident. Consider inequality 
$$\frac{e^{2kt}+1}{2}>\frac{e^{2kt}-1}{2k},$$
which is equivalent to inequality  $k\left((k-1)e^{2kt}+k+1\right)>0.$ Obviously, it holds for all $t>0$ if $|k|\ge 1,$ whence \ref{zeroonehalf5} follows. Let $0<k<1$. Then the latter inequality is equivalent to $(1-k)e^{2kt}<k+1$, or $0< t<\frac{1}{2k}\log\left(\frac{k+1}{1-k}\right).$ Similarly, for $-1<k<0$ we get the equivalent inequality $(1-k)e^{2kt}>k+1$, thus,
$0< t<\frac{1}{2k}\log\left(\frac{k+1}{1-k}\right)$, whence~\ref{zeroonehalf3} follows. 
\end{proof}

\begin{remark}
    Consider in more detail the process 
${X^{H,-1}}=\{X^{H,-1}_t,\, t\ge 0\}$  that has the form 
$$
X^{H,-1}_{t}=\int_0^te^{-s}dB_s^H,\quad H\in[0, 1].
$$
Its variance can be calculated according to Lemma \ref{lem:covar}, replacing everywhere exponent $e^s$ for $e^{-s}$ for any argument $s\ge 0,$ and it equals 
\begin{multline}\label{eq:varsk-}
\mathsf{V}(H, -1, t)
=\mathsf{E} \left(X^{H,-1}_{t}\right)^2
\\
=e^{-t}t^{2H}+\frac{1}{2}\int_0^te^{-z} z^{2H}dz
  -\frac{1}{2}e^{-2t} \int_0^t e^{u} u^{2H}du
=\mathsf{V}(H, 1, t)e^{-2t}.
\end{multline}

Also, $X^{H,-1}$ can be extended to $H=0$ and $H=1$ similarly to Subsections \ref{extendH=1} and \ref{extendH=0}, and taking  into account  \eqref{eq:varsk0}, \eqref{eq:varsk1} and \eqref{eq:varsk-},  from where we obtain the following formulas for variances:
\begin{gather*}
\mathsf{V}(0, -1, t) = \frac{e^{-2t}+1}{2}=e^{-2t}\mathsf{V}(0, \,t),
\quad \mathsf{V} (1/2, -1, t)= \frac{1-e^{-2t}}{2}=e^{-2t}\mathsf{V} (1/2, \,t),\\
\mbox{and} 
\quad \mathsf{V}(1, -1, t)= \left(1-e^{-t}\right)^2=e^{-2t}\mathsf{V}(1, \,t).
\end{gather*}
Therefore the behavior of $\mathsf{V}(H, -1, t)$ in $H$ is described by Theorems  \ref{IBYMtheorem_2}, \ref{IBYMsimsim-3} and  Corollary \ref{IBYMcorcor-2}.
Moreover, the process $Y_t=e^tX^{H,-1}_t=\int_0^te^{t-s}dB_s^H,\, t\ge 0$, that is often identified as the    Ornstein--Uhlenbeck fractional process,  has the same variance as $X_t^H$. This fact is not surprising if we remember that fractional Brownian motion has stationary increments. Note, however, that the entropy of $X^{H,-1}$ as the function of time has opposite behavior for $H\ge \frac12$ and $H<\frac12$: for $H\ge \frac12$ it increases in time while for   $H<\frac12$ decreases. Indeed, it is obvious for $H=\frac12$. Let $H\neq \frac12$. Then  \begin{gather*} \frac{\partial \mathsf{V}(H, -1, t)}{\partial t}=e^{-2t}\left(-e^tt^{2H}+2He^tt^{2H-1}+\int_0^te^zz^{2H}dz\right),
 \end{gather*}
 and the function in the brackets equals zero at zero and has a derivative given by $2H( 2H-1)e^tt^{2H-1}$. So, it is strictly positive for $H>\frac12$ and strictly negative for $H < \frac12$. 
It is worth mentioning that the formulas for the variance and covariance of the fractional Ornstein--Uhlenbeck process were derived in \cite{KMR2017} and \cite{MPRYT2018}, respectively. The results presented in Lemma~\ref{lem:covar} are consistent with these findings.
 \end{remark}

\appendix
\section{}\label{IBYMappen}

\begin{lemma}\label{lemlem1} The following improper integrals are well defined and have such signs:   $$\int_0^\infty e^{-z}\log z\,dz<0,\, \int_0^\infty e^{-z}z\log z\,dz>0,\,\int_0^\infty e^{-z}z^2\log z\,dz>0.$$
\end{lemma}
\begin{proof} It is evident that all three integrals exist. Note  that $\int_0^{1} \log{z}\,dz=-1$ and for $z>1$ we have inequality $\log z<z-1$. Therefore, integrating by parts, we get that
\begin{align*}
    \int_0^{\infty}e^{-z}\log{z}\,dz
    &=\int_0^{1}e^{-z}\log{z}\,dz+\int_1^{\infty}e^{-z}\log{z}\,dz
    \\
    &<e^{-1}\int_0^{1}\log{z}\,dz+\int_1^{\infty}e^{-z}(z-1)\,dz
    =-e^{-1}+e^{-1}=0.
\end{align*}
Next, observe that $\int_0^{1}e^{-z}z^2\log{z}\,dz > \int_0^{1}e^{-z}z\log{z}\,dz$ and $\int_1^{\infty}e^{-z}z^2\log{z}\,dz>\int_1^{\infty}e^{-z}z\log{z}\,dz$.  Consequently, it suffices  to prove that  $\int_0^\infty e^{-z}z\log z\,dz>0$.  Applying inequality $z\log z\ge -e^{-1},\,0<z\le 1$ and integrating by parts on the interval $[1, \infty)$, we   proceed as follows:
 \begin{align*}
\int_0^\infty e^{-z}z\log z\,dz &= \int_0^1 e^{-z}z\log z\,dz+\int_1^\infty e^{-z}z\log z\,dz\\
&> -e^{-1}(1-e^{-1})+\int_1^\infty  e^{-z}(z+1)z^{-1}\,dz \\
&= e^{-2}+\int_1^\infty  e^{-z} z^{-1}\,dz>0.  
\end{align*}
     Lemma is proved.
\end{proof}

\end{document}